\documentclass[11pt, reqno]{amsart}
\usepackage{graphicx,indentfirst}
\usepackage{amscd}

\usepackage[all]{xy}
\usepackage{amsmath}
\usepackage{amsfonts,amssymb}
\usepackage{amsthm}
\usepackage{latexsym}
\usepackage[american]{babel}
\usepackage{mathrsfs}
\usepackage{times}
\usepackage{tikz}
\usepackage{url}
\usepackage{verbatim}
\usetikzlibrary{matrix,arrows,decorations.pathmorphing}
\usepackage[letterpaper,left=2.5cm,right=2.5cm,bottom=3cm,top=3.5cm]{geometry}
\RequirePackage[T1]{fontenc}

\begin{document}

\title{Scalar extensions of categorical resolutions of singularities}
\author{Zhaoting Wei}
\address{Kent State University at Geauga, 14111 Claridon-Troy Road, Burton, OH 44021}
\email{zwei3@kent.edu}

\newcommand{\coh}{\text{coh}}
\newcommand{\Qcoh}{\text{Qcoh}}
\newcommand{\perf}{\text{perf}}
\newcommand{\Perf}{\text{Perf}}
\newcommand{\Hom}{\text{Hom}}
\newcommand{\End}{\text{End}}
\newcommand{\Pic}{\text{Pic}}
\newcommand{\red}{\text{red}}
\newcommand{\Spec}{\text{Spec}}
\newcommand{\Frac}{\text{Frac}}
\newcommand{\rD}{\mathrm{D}}
\newcommand{\LInd}{\text{LInd}}
\newcommand{\Ch}{\text{Ch}}

\newtheorem{thm}{Theorem}[section]
\newtheorem{lemma}[thm]{Lemma}
\newtheorem{prop}[thm]{Proposition}
\newtheorem{coro}[thm]{Corollary}
\newtheorem{conj}{Conjecture}
\theoremstyle{definition}\newtheorem{defi}{Definition}[section]
\theoremstyle{remark}\newtheorem{eg}{Example}
\theoremstyle{remark}\newtheorem{rmk}{Remark}
\theoremstyle{remark}\newtheorem{ques}{Question}

\maketitle
%\begin{abstract}

%\end{abstract}

%\tableofcontents
\begin{abstract}
Let $X$ be a quasi-compact, separated scheme over a field $k$ and we can consider the categorical resolution of singularities of $X$.  In this paper let $k^{\prime}/k$ be a field extension and we study the scalar extension of a categorical resolution of singularities of $X$ and we show how it gives a categorical resolution of the base change scheme $X_{k^{\prime}}$. Our construction involves the scalar extension of derived categories of DG-modules over a DG algebra. As an application we use the technique of scalar extension developed in this paper to prove the non-existence of full exceptional collections of categorical resolutions for a projective curve of genus $\geq 1$ over a non-algebraically closed field.
\end{abstract}

\section{Introduction}
For a scheme $X$ over a field $k$, the derived category $\rD(X)$ of quasi-coherent $\mathcal{O}_X$-modules plays an important role in the study of the geometry of $X$. In particular, a \emph{categorical resolution} of singularities of $X$ is defined to be a smooth triangulated category $\mathscr{T}$ together with an adjoint pair $\pi^*: \rD(X)\rightleftarrows \mathscr{T}: \pi_*$ which satisfies certain properties. See \cite{kuznetsov2014categorical} or Definition \ref{defi: categorical resolution} below for details.

On the other hand, base change techniques are also ubiquitous in algebraic geometry. In \cite{sosna2014scalar}, a theory of scalar extensions of triangulated categories has been developed and applied to derived categories of varieties.

In this paper we   define and study the scalar extension of categorical resolutions. The difficulty is to find the scalar extensions of the adjoint pair $(\pi^*,\pi_*)$. To solve this problem we modify the definition of categorical resolution: inspired by \cite{lunts2010categorical}, we define an \emph{algebraic categorical resolution} of $X$ to be a triple $(A,B,T)$ where $A$ is a differential graded (DG) algebra such that $\rD(X)\simeq \rD(A)$, $B$ is a smooth DG algebra and $T$ is an $A$-$B$ bimodule which satisfies certain properties. See Definition \ref{defi: categorical resolution our form} below for more details. In some important cases, which include the cases we are most interested in, these two definitions are equivalent. For the comparison of different definitions of categorical resolution see Proposition \ref{prop: a categorical resolution gives a categorical resolution of dga} below.

The advantage of algebraic categorical resolution is that it is compatible with base field extensions. One of the main results in this paper is the following proposition.

\begin{prop}[See Proposition \ref{prop: scalar extension of a categorical resolution is a categocial resolution} below]\label{prop: scalar extension of a categorical resolution is a categocial resolution introduction}[See Proposition \ref{prop: scalar extension of a categorical resolution is a categocial resolution} below]
Let $X$ be a projective variety over a field $k$. If  $(A,B,T)$ is an algebraic categorical resolution of $X$, then $(A_{k^{\prime}},B_{k^{\prime}},T_{k^{\prime}})$ is an algebraic categorical resolution of the base change variety $X_{k^{\prime}}$.
\end{prop}

As an application we study the categorical resolution of projective curves $X$ over a non-algebraically closed field $k$. Using the technique of scalar extension we obtain the following theorems which generalize the main results in \cite{wei2016full}.

\begin{thm}[See Theorem \ref{thm: full exceptional collection non-alg closed field} below]\label{thm: full exceptional collection non-alg closed field introduction}
Let $X$ be a projective curve over a  field $k$. Then $X$ has a categorical resolution which admits a full exceptional collection if and only if the geometric genus of $X$ is $0$.
\end{thm}

\begin{thm}[See Theorem \ref{thm: tilting object non-alg closed field} below]\label{thm: tilting object non-alg closed field introduction}
Let $X$ be a projective curve with geometric genus $\geq 1$ over a  field $k$ and  $(\mathscr{T},\pi^*,\pi_*)$ be a categorical resolution of $X$. Then $\mathscr{T}^c$ cannot have a tilting object, moreover there cannot be a finite dimensional $k$-algebra $\Lambda$ of finite global dimension such that
$$
\mathscr{T}^c\simeq \rD^b(\Lambda\text{-mod}).
$$
\end{thm}

This paper is organized as follows: In Section \ref{section: prelim} we quickly review  triangulated categories, DG categories and DG algebras. In Section \ref{section: Categorical resolution of singularities} we review and compare different definitions of categorical resolutions. In Section \ref{section: Scalar extensions of derived categories of DG algebras} we study the scalar extension of derived categories of DG algebras and in Section \ref{section: Scalar extensions of categorical resolutions} we study the scalar extension of categorical resolutions. In Section \ref{section: full exceptional collections of categorical resolutions} we use the technique of scalar extension to study categorical resolutions of projective curves over a non-algebraically closed field.

\section*{Acknowledgements}
First the author wants to thank Valery Lunts for introducing him to this topic and for helpful comments on the previous version of this paper. The author also wants to thank Pieter Belmans, Francesco Genovese, Adeel Khan, Jeremy Rickard, and Pedro Tamaroff for answering questions related to this paper. Part of this paper was finished when the author was visiting the Institut des Hautes \'{E}tudes Scientifiques, France and he wants to thank IHES for its hospitality and its ideal working environment.

\section{Preliminaries}\label{section: prelim}
\subsection{Review of  some concepts on triangulated categories}\label{section: review of triangulated categories}
Let $\mathcal{T}$ be a triangulated category. $\mathcal{T}$ is called \emph{cocomplete} if it has arbitrary direct sums. An object $E$ of a cocomplete triangulated category $\mathcal{T}$ is called \emph{compact} if the functor  $\Hom_{\mathcal{T}}(E,-)$ preserves arbitrary direct sums. Let $\mathcal{T}^c$ denote the full triangulated subcategory of $\mathcal{T}$ consisting of compact objects.

Let $\mathcal{I}$ be a set of objects of $\mathcal{T}$. We say $\mathcal{I}$ \emph{generates} $\mathcal{T}$ if for any object $N$ of $\mathcal{T}$, $\Hom_{\mathcal{T}}(E,N[i])=0$ for any $E\in \mathcal{I}$ and $i\in \mathbb{Z}$ implies $N=0$. We say a cocomplete triangulated category $\mathcal{T}$ is \emph{compactly generated} if it is generated by a set of compact objects. An object $E$ of $\mathcal{T}$ is called a generator of $\mathcal{T}$ if the set $\{E\}$ generates $\mathcal{T}$.

We have the following well-known result.

\begin{lemma}\label{lemma: the subcategory contains compact generator is the whole category}
Let $\mathcal{T}$ be a cocomplete triangulated category. If a set of objects $\mathcal{E}\subset \mathcal{T}^c$ generates $\mathcal{T}$, then $\mathcal{T}$ coincides with the smallest strictly full triangulated subcategory
of $\mathcal{T}$ which contains $\mathcal{E}$ and is closed under direct sums. Recall a subcategory is strictly full if it is full and closed under isomorphism.
\end{lemma}
\begin{proof}
See the proof of \cite{rouquier2008dimensions} Theorem 4.22.
\end{proof}

For a  scheme $X$, let $\rD(X)$ be the unbounded derived category of complexes of $\mathcal{O}_X$-modules with quasi-coherent cohomologies. Moreover if $X$ is noetherian let $\rD^{\text{perf}}(X)$ be the derived category of perfect complexes on $X$ and  $\rD^b(\coh(X))$ be the derived category of bounded complexes of $\mathcal{O}_X$-modules with coherent cohomologies.

\begin{rmk}\label{rmk: different D qcoh}
We could also consider $\rD(\Qcoh(X))$, the derived category of unbounded complexes of quasi-coherent sheaves on $X$. There exists a canonical triangulated functor $i: \rD(\Qcoh(X))\to \rD(X)$. For a general scheme $X$, the functor $i$ needs not to be an equivalence. However if $X$ is noetherian or quasi-compact and separated, then $i$ must be an triangulated equivalence, see \cite[Tag 09T1]{stacks-project} Proposition 35.7.3 and \cite{bokstedt1993homotopy} Corollary 5.5. Since all schemes we study in this paper are either noetherian or quasi-compact and separated, we could identify $\rD(\Qcoh(X))$ and $\rD(X)$ and we will restrict ourselves to $\rD(X)$ from now on.
\end{rmk}

We have the following result on compact objects and compact generators of $\rD(X)$.

\begin{prop}[\cite{bondal2003generators} Theorem 3.1.1]\label{prop: compact objects and compact generators in D(X)}
Let $X$ be a quasi-compact quasi-separated scheme. Then
\begin{enumerate}
\item The compact objects in $\rD(X)$ are precisely the perfect complexes;
\item $\rD(X)$ has a compact generator.
\end{enumerate}
\end{prop}
\begin{proof}
See the proof of \cite{bondal2003generators} Theorem 3.1.1.
\end{proof}

\subsection{Review of some concepts on DG categories}
Most of results in this subsection could be found in \cite{kuznetsov2014categorical} Section 3. For a comprehensive introduction to DG categories see \cite{keller2006differential}.

Let $\mathcal{D}$ be a DG category over $k$ and let $H^0(\mathcal{D})$ denote the homotopy category of $\mathcal{D}$. Let $Z^0(\mathcal{D})$ be the category which has the same objects as $\mathcal{D}$ and whose morphisms from an object $x$ to an object $y$ are the degree $0$ closed morphisms in $\mathcal{D}(x, y)$.

A right $\mathcal{D}$-module is a DG functor $\mathcal{D}^{op}\to \Ch(k)$. The category of right $\mathcal{D}$-modules has a natural DG structure and let $\mathcal{M}_{dg}(\mathcal{D})$ denote the DG category of right $\mathcal{D}$-modules. Moreover, a DG-modules $M$ over $\mathcal{D}$ is called acyclic if for any object $X\in \mathcal{D}$, the complex $M(X)$ is acyclic. Let $\text{Acycl}(\mathcal{D})$ denote the full DG subcategory of $\mathcal{M}_{dg}(\mathcal{D})$ which consists of acyclic DG-modules over $\mathcal{D}$.

For a DG category $\mathcal{D}$, the derived category of DG-modules over $\mathcal{D}$ is defined to be the Verdier quotient
$$
\rD(\mathcal{D}):=H^0(\mathcal{M}_{dg}(\mathcal{D}))/H^0(\text{Acycl}(\mathcal{D})).
$$

\begin{rmk}
We can prove that $\rD(\mathcal{D})$ is a cocomplete, compactly generated triangulated category.
\end{rmk}

\begin{defi}\label{defi: h-injective projective DG modules}
A DG-module $P$ is called \emph{h-projective} if, for any acyclic DG-module $M$, the complex $\Hom_{\mathcal{M}_{dg}(\mathcal{D})}(P,M)$ is acyclic. Dually, a DG-module $I$ is called \emph{h-injective} if, for any acyclic DG-module $M$, the complex $\Hom_{\mathcal{M}_{dg}(\mathcal{D})}(M,I)$ is acyclic.

We denoted the full DG subcategory of $\mathcal{M}_{dg}(\mathcal{D})$ consisting of h-projective DG-modules by h-proj$(\mathcal{D})$ and the full DG subcategory of $\mathcal{M}_{dg}(\mathcal{D})$ consisting of h-injective DG-modules by h-inj$(\mathcal{D})$.
\end{defi}

\begin{rmk}
We can prove that $\rD(\mathcal{D})\simeq H^0(\text{h-proj}(\mathcal{D}))\simeq H^0(\text{h-inj}(\mathcal{D}))$.
\end{rmk}

There is a standard Yoneda embedding $\mathcal{D}\to \mathcal{M}_{dg}(\mathcal{D})$ given by
$$
x\mapsto h_x:=\mathcal{D}(-,x).
$$

\begin{defi}\label{defi: representable DG modules}
A DG-module $M$ is called \emph{representable} if  it is isomorphic in  $Z^0(\mathcal{M}_{dg}(\mathcal{D}))$ to an object of the form $h_x$ for some $x\in \mathcal{D}$.

Moreover, a DG-module $M$ is called \emph{quasi-representable} if  it is isomorphic in  $\rD(\mathcal{D})$ to an object of the form $h_x$ for some $x\in \mathcal{D}$.
\end{defi}

For two DG categories $\mathcal{C}$ and $\mathcal{D}$, a $\mathcal{C}$-$\mathcal{D}$ DG-bimodule is a right DG module over $\mathcal{D}^{op}\otimes \mathcal{C}$, i.e. a DG functor $\mathcal{D}\otimes \mathcal{C}^{op}\to \Ch(k)$. For a $\mathcal{C}$-$\mathcal{D}$ DG-bimodule $T$, the derived tensor product defines a functor
$$
(-)\otimes^{\mathbf{L}}_{\mathcal{C}}T: \rD(\mathcal{C})\to \rD(\mathcal{D}).
$$

\begin{defi}\label{defi: quasi-functor}
A $\mathcal{C}$-$\mathcal{D}$ DG-bimodule $T$ is called a \emph{quasi-functor} if for any $x\in \mathcal{C}$, the object $T(x,-)\in \mathcal{M}_{dg}(\mathcal{D})$ is quasi-representable. It is clear that a quasi-functor defines a functor $H^0(\mathcal{C})\to H^0(\mathcal{D})$.
\end{defi}

For a DG category $\mathcal{D}$ we could consider $\mathcal{D}$-$\mathcal{D}$ DG-bimodules and in particular $\mathcal{D}$ itself could be considered as the diagonal bimodule
$$
\mathcal{D}(X,Y)=\Hom_{\mathcal{D}}(Y,X)\in \Ch(k).
$$

\begin{defi}\label{defi: smooth DG cat and smooth triangulated cat}
A DG category $\mathcal{D}$ is called \emph{smooth}  if the diagonal
bimodule  $\mathcal{D}$ is a perfect bimodule. In other words, if  $\mathcal{D}$ is a direct summand (in the derived category of $\mathcal{D}$-$\mathcal{D}$ DG-bimodules)
of a bimodule obtained from quasi-representable bimodules by finite number of shifts and
cones of closed morphisms.

Moreover, a triangulated category $\mathcal{T}$ is called smooth if there exists a smooth DG category $\mathcal{D}$ such that $\mathcal{T}$ is triangulated equivalent to $\rD(\mathcal{D})$.
\end{defi}

\begin{rmk}\label{rmk: smooth variety and smooth triangulated category}
It is well known  that if $X$ is a smooth variety, then the derived category
$\rD(X)$ is a smooth triangulated category, see \cite{toen2007moduli}.
\end{rmk}

\subsection{Review of some concepts on DG algebras}
Let $A$ be a DG algebra over $k$. We could consider $A$ as a DG category with one  object. Therefore most of  concepts for DG categories could be defined for DG algebras without any changes. For example, we could define right DG-modules, the derived category, h-projective and h-injective modules and smoothness for a DG algebra.

For a DG algebra $A$, let $\rD(A)$ be the unbounded derived category of complexes of  right DG A-modules and Perf$(A)$ be the full subcategory of perfect complexes of  right $A$-modules.

For later application we recall the following characterization of compact objects in $\rD(A)$.
\begin{lemma}\label{lemma: characterization of compact objects}[\cite[Tag 09QZ]{stacks-project} Proposition 22.27.4]
Let $A$ be a DG algebra. Let $E$ be an
object of $\rD(A)$. Then the following are equivalent
\begin{enumerate}
\item $E$ is a compact object;
\item $E$ is a direct summand of an object of $\rD(A)$
which is represented by a differential graded module $P$ which
has a finite filtration $F_\bullet$ by differential graded submodules
such that $F_iP/F_{i - 1}P$ are finite direct sums of shifts of $A$, i.e. $E$ is an object of Perf$(A)$.
\end{enumerate}
\end{lemma}
\begin{proof}
See the proof of \cite[Tag 09QZ]{stacks-project} Proposition 22.27.4.
\end{proof}

The following definition plays a significant role in the constructions in this paper.
\begin{defi}\label{defi: dg morita equivalence}
Let $A$ and $B$ be two DG algebras over $k$. Then we call $A$ and $B$ are DG Morita equivalent if there exists an $A$-$B$ bimodule $T$ such that the derived tensor product functor $(-)\otimes^{\mathbf{L}}_A T$ induces a triangulated equivalence $\rD(A)\overset{\sim}{\to}\rD(B)$.
\end{defi}

\begin{rmk}
In a recent preprint \cite{rizzardo2017note}, Rizzardo and Van den Bergh constructed two $A_{\infty}$-algebras $F$ and $F_{\eta}$ over a field $k$ such that $\Perf(F)$ and $\Perf(F_{\eta})$ are triangulated equivalent but their $A_{\infty}$-enhancements are not $A_{\infty}$-equivalent. It could be deduced from this fact that there exist two DG algebras $A$ and $B$ over $k$ such that $\rD(A)\simeq \rD(B)$ but $A$ and $B$ are not DG Morita equivalent.
\end{rmk}

The following proposition is used to connect a DG category to the derived category of a DG algebra.

\begin{prop}[\cite{lunts2016new} Proposition B.1 (b)]\label{prop: equivalence between dg category and module cateogy of dg algebra}
The DG categories are over an arbitrary ground ring. Let $\mathcal{C}$ be a DG category with a full pretriangulated DG subcategory $\mathcal{I}$. Let $z: P\to I$ be a closed degree zero morphism in $\mathcal{C}$ with $I\in \mathcal{I}$ such that $z^*: \Hom_{\mathcal{C}}(I,J)\to \Hom_{\mathcal{C}}(P,J)$ is an quasi-isomorphism for all $J\in \mathcal{I}$, and let $B$ be a DG algebra together with a morphism $\beta: B\to \End_{\mathcal{C}}(P)$ of DG algebras such that the composition $B\overset{\beta}{\to}\End_{\mathcal{C}}(P)\overset{z_{*}}{\to}\Hom_{\mathcal{C}}(P,I)$ is a quasi-isomorphism. If $\mathcal{I}$ has all coproducts and $I$ is a compact generator of $H^0(\mathcal{I})$, then the functor
$$
\beta^*\circ \Hom(P,-): H^0(\mathcal{I})\to \rD(B)
$$
is an equivalence of triangulated categories.
\end{prop}
\begin{proof}
See the proof of  \cite{lunts2016new} Proposition B.1 (b).
\end{proof}
\section{Categorical resolution of singularities}\label{section: Categorical resolution of singularities}
\subsection{Review of definitions of categorical resolutions}\label{section: Review of definitions}
In \cite{kuznetsov2008lefschetz} and \cite{kuznetsov2014categorical} the  \emph{categorical resolution of singularities} has been defined and studied.

\begin{defi}[\cite{kuznetsov2008lefschetz} Definition 3.2 or \cite{kuznetsov2014categorical} Definition 1.3]\label{defi: categorical resolution}
A categorical resolution of a scheme $X$ is a smooth, cocomplete,
compactly generated, triangulated category $\mathscr{T}$ with an adjoint pair of triangulated
functors
$$
\pi^*: \rD(X)\to \mathscr{T} \text{ and } \pi_*: \mathscr{T}\to \rD(X)
$$
such that
\begin{enumerate}
\item $id\simeq \pi_*\circ \pi^*$;
\item both $\pi_*$ and $\pi^*$ commute with arbitrary direct sums;
\item $\pi_*(\mathscr{T}^c)\subset \rD^b(\coh(X))$ where $\mathscr{T}^c$ denotes the full subcategory of $\mathscr{T}$ which consists of compact objects.
\end{enumerate}
\end{defi}

\begin{rmk}
Note that Condition (1) implies that $\pi^*$ is fully faithful and Condition (2) implies that $\pi^*(\rD^{\text{perf}}(X))\subset \mathscr{T}^c$.
\end{rmk}

\begin{rmk}\label{rmk: coherent is dual to perfect}
Let $X$ be a projective variety over an algebraically closed field $k$. By \cite{rouquier2008dimensions} Corollary 7.51, an object $\mathcal{F}$ belongs to the subcategory $\rD^b(\coh(X))$ if and only if
$$
\bigoplus_{n\in \mathbb{Z}}\Hom_{\rD(X)}(\mathcal{E},\mathcal{F}[n])
$$
 is finite dimensional for any $\mathcal{E}\in \rD^{\text{perf}}(X)$. Therefore Condition (3) in Definition \ref{defi: categorical resolution} is equivalent to the following condition: For any $\mathcal{E}\in \rD^{\text{perf}}(X)$ and $\mathcal{F}\in \mathscr{T}^c$, the vector space $\bigoplus_n\Hom_{\mathscr{T}}(\mathcal{E},\pi_*\mathcal{F}[n])$ is finite dimensional. For later applications we also notice that it is equivalent to require $\bigoplus_n\Hom_{\mathscr{T}}(\pi^*\mathcal{E},\mathcal{F}[n])$ to be finite dimensional.
\end{rmk}

In \cite{kuznetsov2014categorical} the existence of categorical resolutions has been proved.

\begin{thm} [\cite{kuznetsov2014categorical} Theorem 1.4]
Any separated scheme $Y$ of finite type over a field of characteristic $0$ has
a categorical resolution.
\end{thm}

\begin{rmk}
In general, the categorical resolution of a scheme is not unique.
\end{rmk}

On the other hand, we notice that in \cite{lunts2010categorical} there is another definition of categorical resolution.

\begin{defi}[\cite{lunts2010categorical} Definition 4.1]\label{defi: categorical resolution dga}
Let $A$ be a DG algebra, A categorical resolution of $\rD(A)$ is a pair $(B,T)$ where $B$ is a smooth DG algebra and $T\in \rD(A^{op}\otimes B)$  such that the restriction of the functor
$$
\theta(-):=(-)\otimes^{\mathbf{L}}_A T;~ \rD(A)\to \rD(B)
$$
to the subcategory Perf$(A)$ is full and faithful.
\end{defi}

\begin{rmk}
It is clear that $\theta=(-)\otimes^{\mathbf{L}}_A T$ commutes with arbitrary direct sum. Moreover its right adjoint $R\Hom_B(T,-)$ commutes with arbitrary direct sum if and only if $T$ is compact when considered as an object in $\rD(B)$. In this case it is also clear that $\theta$ maps Perf$(A)$ to Perf$(B)$ and hence by  Lemma \ref{lemma: full faithfull from Perf to whole} below, $\theta: \rD(A)\to\rD(B)$ is fully faithful.
\end{rmk}

\begin{lemma}\label{lemma: full faithfull from Perf to whole}
Let $A$ and $B$ be DG algebras and let $F: \rD(A)\to \rD(B)$ be a triangulated functor with the following
properties
\begin{enumerate}
\item $F(\text{Perf}(A))\subset \text{Perf}(B)$;
\item The restriction of $F$ to Perf$(A)$ is fully faithful;
\item $F$ preserves direct sums.
\end{enumerate}
Then $F$ is fully faithful.
\end{lemma}
\begin{proof}
See the proof of  \cite{lunts2010categorical} Lemma 2.13.
\end{proof}

\subsection{Comparison of definitions}\label{section: Comparison of definitions}
Conceptually Definition \ref{defi: categorical resolution} and Definition \ref{defi: categorical resolution dga} are very similar and we want to find the relation between them.

First of all, for a quasi-compact and quasi-separated scheme $X$, the derived category $\rD(X)$ has a compact generator $\mathcal{E}$ hence we have an equivalence of triangulated categories.
$$
\Phi: \rD(X)\overset{\sim}{\to} \rD(A)
$$
where $A:=R\text{Hom}_X(\mathcal{E},\mathcal{E})$. We notice that $\Phi$ maps $\rD^{\text{perf}}(X)$ to Perf$(A)$.

\begin{rmk}\label{rmk: compactly generated category does not have a compact generator}
We know that in Definition \ref{defi: categorical resolution} the triangulated category $\mathscr{T}$ is compactly generated. However, in general a triangulated category $\mathscr{T}$ is compactly generated does not imply that it has a single compact generator. For counterexamples see \cite{hall2015algebraic} Theorem A (2)(a).
\end{rmk}

Now we assume the triangulated category $\mathscr{T}$ in Definition \ref{defi: categorical resolution} has a compact generator, then we also have a DG algebra $B$ and an equivalence
$$
\Psi: \mathscr{T} \overset{\sim}{\to} \rD(B)
$$
which maps $\mathscr{T}^c$ to Perf$(B)$.

As a result we could reformulate Definition \ref{defi: categorical resolution} as follows.

\begin{defi}[The auxiliary definition of categorical resolution]\label{defi: categorical resolution modified form}
Let $X$ be a projective variety over a field $k$. Let $A$ be a DG algebra such that $\rD(X) \simeq \rD(A)$. Then a categorical resolution of $X$ is a smooth DG algebra $B$ with an adjoint pair of triangulated functors
$$
\pi^*: \rD(A)\to \rD(B) \text{ and } \pi_*: \rD(B)\to \rD(A)
$$
such that
\begin{enumerate}
\item $\pi_*\circ \pi^*\simeq id$;
\item both $\pi_*$ and $\pi^*$ commute with arbitrary direct sums;
\item For any $E\in \text{Perf}(A)$ and $F\in \text{Perf}(B)$, the vector space
$$
\bigoplus_{n\in \mathbb{Z}}\Hom_{\rD(B)}(\pi^*E,F[n])
$$
is finite dimensional.
\end{enumerate}
\end{defi}

\begin{rmk}
In the light of Remark \ref{rmk: coherent is dual to perfect}, Condition (3) in Definition \ref{defi: categorical resolution modified form} is equivalent to Condition (3) in Definition \ref{defi: categorical resolution}. Moreover, by the definition of perfect complexes, Condition (3) in Definition \ref{defi: categorical resolution modified form} could be replaced by the following weaker form.
\begin{enumerate}
\item[($\text{3}^{\prime}$)] The vector space
$$
\bigoplus_{n\in \mathbb{Z}}\Hom_{\rD(B)}(\pi^*A,B[n])
$$
is finite dimensional.
\end{enumerate}

Nevertheless we also notice that we need $X$ to be a projective variety to make sure that Condition (3) (or ($\text{3}^{\prime}$)) in Definition \ref{defi: categorical resolution modified form} could replace  Condition (3) in Definition \ref{defi: categorical resolution}. It may fail for general schemes.
\end{rmk}

We want to find the relation between Definition \ref{defi: categorical resolution modified form} and Definition \ref{defi: categorical resolution dga}. The difficulty is to show that the functor $\pi^*: \rD(A)\to \rD(B)$ is given by  a derived tensor product $(-)\otimes^{\mathbf{L}}_A T$ for a $T\in \rD(A^{op}\otimes B)$. It is a problem because we know that there exist DG algebras   which are not DG Morita equivalent but their derived categories are equivalent as triangulated categories.

Nevertheless, we can obtain another DG algebra $\widetilde{A}$ with $\rD(A)\simeq \rD(\widetilde{A})$ together with an $\widetilde{A}$-$B$ bimodule $T$ which gives a DG Morita equivalence. Actually we have the following proposition.

\begin{prop}\label{prop: another DG algebra with the same derived category and a bimodule}
Let $A$ and $B$ be DG algebras and $\pi^*: \rD(A)\leftrightarrows \rD(B): \pi_*$ be an adjoint pair of triangulated functors which satisfies Condition (1),(2),(3) in Definition \ref{defi: categorical resolution modified form}. Let $T:=\pi^*A\in \rD(B)$ and
$$
\widetilde{A}:=R\Hom_B(T,T).
$$
Then we can build a triangulated equivalence
$$
\phi: \rD(A)\overset{\sim}{\to} \rD(\widetilde{A}).
$$
\end{prop}
\begin{proof}
The proof is essentially the same as that of \cite{lunts2010categorical} Proposition 2.6. We give a proof here for completeness.

First we define $\phi$ to be the composition
$$
\phi: \rD(A)\xrightarrow{\pi^*}\rD(B)\xrightarrow{R\Hom_B(T,-)}\rD(\widetilde{A}).
$$
$\phi$ commutes with direct sums because $\pi^*$ does and because $\pi^*$ is fully faithful.

Then we prove that $\phi$ is fully faithful. Let $\mathcal{D}_0\subset  \rD(A)$ be the strictly full triangulated subcategory consisting of objects $M$ such that the natural map
$$
\Hom_{ \rD(A)}(A,M[i])\to \Hom_{\rD(\widetilde{A})}(\phi(A),\phi(M)[i])
$$
is an isomorphism for any integer $i$. Since $\phi(A)=\widetilde{A}$ and $H^i(A)\cong H^i(\widetilde{A})$ for any $i$, it is clear that $A\in \mathcal{D}_0$. Moreover since $\phi$ commutes with arbitrary direct sums we see that $\mathcal{D}_0$ is closed under arbitrary direct sums. By Lemma \ref{lemma: the subcategory contains compact generator is the whole category}, $\mathcal{D}_0=\rD(A)$. Similarly let  $\mathcal{D}_1\subset  \rD(A)$ be the strictly full triangulated subcategory consisting of objects $N$ such that the natural map
$$
\Hom_{ \rD(A)}(N,M)\to \Hom_{\rD(\widetilde{A})}(\phi(N),\phi(M))
$$
is an isomorphism for any object $M\in \rD(A)$. By the above reasoning $A\in \mathcal{D}_1$ and it is clear that $\mathcal{D}_1$ is closed under arbitrary direct sums. Again by Lemma \ref{lemma: the subcategory contains compact generator is the whole category} we get $\mathcal{D}_1= \rD(A)$, i.e. $\phi$ is fully faithful.

For the essential surjectivity, since $\phi(A)=\widetilde{A}$ we know that $\phi(\rD(A))\supset \text{Perf}(\widetilde{A})$. In addition, it is clear that $\phi(\rD(A))$ is closed under direct sums. Again by Lemma \ref{lemma: the subcategory contains compact generator is the whole category} $\phi$ is essentially surjective.
\end{proof}

\begin{rmk}
From the construction it is clear that $H^i(A)\cong H^i(\widetilde{A})$ for any $i\in \mathbb{Z}$. However, a priori there is no quasi-isomorphism between $A$ and $\widetilde{A}$ therefore the fact that $\rD(A)\simeq \rD(\widetilde{A})$ is not trivial.
\end{rmk}

Now we give our version of definition of categorical resolution of singularities.

\begin{defi}[Algebraic categorical resolution]\label{defi: categorical resolution our form}
Let $X$ be a projective variety over a  field $k$. Then an algebraic categorical resolution of $X$ is a triple $(A,B,T)$ where $A$ is a DG algebra such that $\rD(X) \simeq \rD(A)$, $B$ is a smooth DG algebra and $T$ is an
$A$-$B$ bimodule such that
\begin{enumerate}
\item $H^i(A)\to \Hom_{\rD(B)}(T,T[i])$ is an isomorphism for any $i\in \mathbb{Z}$;
\item $T$ defines a compact object in $\rD(B)$;
\item $\bigoplus_i \Hom_{\rD(B)}(T,B[i])$ is finite dimensional.
\end{enumerate}
\end{defi}

It is clear that a triple $(A,B,T)$ as in Definition \ref{defi: categorical resolution our form} gives a triple $(\mathscr{T},\pi^*,\pi_*)$ as in Definition \ref{defi: categorical resolution}. For the other direction of implication we have the following proposition.

\begin{prop}\label{prop: a categorical resolution gives a categorical resolution of dga}
Let $X$ be a projective variety over a field $k$ and $(\mathscr{T},\pi^*,\pi_*)$ be a categorical resolution of $X$ in the sense of Definition \ref{defi: categorical resolution}. If $\mathscr{T}$ has a compact generator, then we have an algebraic categorical resolution $(A,B,T)$ of $X$ in the sense of Definition \ref{defi: categorical resolution our form} such that $\rD(B)\simeq \mathscr{T}$.
\end{prop}
\begin{proof}
Since $\mathscr{T}$ has a compact generator we can find a DG algebra $B$ such that $\rD(B)\simeq \mathscr{T}$. We use the $(\widetilde{A},T)$ in Proposition \ref{prop: another DG algebra with the same derived category and a bimodule} as the $(A,T)$ in Definition \ref{defi: categorical resolution our form}. It is clear that they satisfy all conditions in Definition \ref{defi: categorical resolution our form}.
\end{proof}

\section{Scalar extensions of categorical resolutions of singularities}\label{section: Scalar extensions of categorical resolutions of singularities}
\subsection{Scalar extensions of derived categories of DG algebras}\label{section: Scalar extensions of derived categories of DG algebras}
In this section we study scalar extensions of categorical resolution. First we notice that the scalar extension of triangulated categories has been studied in \cite{sosna2014scalar}. However, we do not know whether the scalar extension as defined in \cite{sosna2014scalar} Definition 9 preserves fully faithful functors hence it is difficult to use that definition directly to study the scalar extension of categorical resolutions.

Nevertheless there is another approach to the scalar extension of triangulated categories which is outlined in \cite{sosna2014scalar} Remark 9.

\begin{defi}\label{defi: scalar extension of derived category of dga}
Let $A$ be a DG algebra over a base field $k$ and we consider the derived category $\rD(A)$. For a field extension $k^{\prime}/k$, we denote $A\otimes_k k^{\prime}$ by $A_{k^{\prime}}$. Then we call $\rD(A_{k^{\prime}})$  the scalar extension of $\rD(A)$.
\end{defi}

We expect that the scalar extension depends on the triangulated category $\rD(A)$ only. In more details we want the following conjecture to be true.

\begin{conj}\label{conj: scalar extension is independent of the choice of algebra}
Let $A$ and $B$ be two DG algebras over $k$ such that we have a triangulated equivalence between derived categories $\rD(A)\overset{\sim}{\to}\rD(B)$. Then for any field extension $k^{\prime}/k$ we have a triangulated equivalence
$$
\rD(A_{k^{\prime}})\overset{\sim}{\to}\rD(B_{k^{\prime}}).
$$
\end{conj}

So far we do not know whether Conjecture \ref{conj: scalar extension is independent of the choice of algebra} is true. Nevertheless, we can prove it in an important special case which is sufficient for our use in algebraic geometry. First we need the following lemmas.

\begin{lemma}\label{lemma: scalar extension of morphisms in derived category another form}
Let $A$ be a DG algebra over $k$ and $k^{\prime}/k$ be a field extension. For an object $F$ in $\rD(A_{k^{\prime}})$ the forgetful functor maps $F$ to an object in $\rD(A)$. Moreover for  any object $E$ in $\rD(A)$ we have
$$
\Hom_{\rD(A_{k^{\prime}})}(E_{k^{\prime}},F)\cong \Hom_{\rD(A)}(E,F).
$$
\end{lemma}
\begin{proof}
Since $A\to A_{k^{\prime}}$ is flat, $(-)\otimes_k k^{\prime}=(-)\otimes_A A_{k^{\prime}}$ gives the derived tensor product functor $\rD(A)\to \rD(A_{k^{\prime}})$ which is left adjoint to the forgetful functor $\rD(A_{k^{\prime}})\to \rD(A)$.
\end{proof}

\begin{lemma}\label{lemma: scalar extension of morphisms in derived category}
Let $A$ be a DG algebra over $k$ and $k^{\prime}/k$ be a field extension. For any objects $E$ in Perf$(A)$ and $F$ in $\rD(A)$, the natural map
$$
\Hom_{\rD(A)}(E,F)\otimes k^{\prime} \to \Hom_{\rD(A_{k^{\prime}})}(E_{k^{\prime}},F_{k^{\prime}})
$$
is an isomorphism.
\end{lemma}
\begin{proof}
The statement is obviously true if $E=A$. Then we use the fact that any perfect complex of $A$-module can be obtained from $A$ by finite direct sums, shifts, direct summands and exact triangles.
\end{proof}

\begin{prop}\label{prop: dg morita equivalence is preserved under scalar extension}
Let $A$ and $B$ be two DG algebras over $k$ with an  $A$-$B$ bimodule $T$ giving a DG Morita equivalence as in Definition \ref{defi: dg morita equivalence}. Then for any field extension $k^{\prime}/k$, the $A_{k^{\prime}}$-$B_{k^{\prime}}$ bimodule $T_{k^{\prime}}$ also gives a DG Morita equivalence between $A_{k^{\prime}}$ and $B_{k^{\prime}}$.
\end{prop}
\begin{proof}
We need the following criterion on DG Morita equivalence.
\begin{lemma}\label{lemma: when a bimodules gives and equivalence}
Let $A$ and $B$ be DG algebras over a field $k$ and $\theta: \rD(A)\to \rD(B)$ be a triangulated functor given by $(-)\otimes^{\mathbf{L}}_A T$ where $T$ is an $A$-$B$ bimodule. Then $\theta$ is a triangulated equivalence   if and only if the following condition holds.
\begin{enumerate}
\item $T$ is compact if considered as an object in $\rD(B)$, i.e. $T$ belongs to the full subcategory Perf$(B)$;
\item If $N$ is an object in $\rD(B)$ and $\Hom_{\rD(B)}(T,N[i])=0$ for any $i\in \mathbb{Z}$, then we have $N=0$;
\item $\Hom_{\rD(B)}(T,T[i])\cong H^i(A)$ for any $i\in \mathbb{Z}$.
\end{enumerate}
\end{lemma}
\begin{proof}[The proof of Lemma \ref{lemma: when a bimodules gives and equivalence}] See \cite[Tag 09S5]{stacks-project} Lemma 22.28.2.
\end{proof}
Now assume $T$ satisfies Condition (1), (2), (3). Then it is clear that $T_{k^{\prime}}$ satisfies Condition (1). Applying Lemma \ref{lemma: scalar extension of morphisms in derived category another form} and  Lemma \ref{lemma: scalar extension of morphisms in derived category}, we see that $T_{k^{\prime}}$ satisfies Condition (2) and (3). This finishes the proof of Proposition \ref{prop: dg morita equivalence is preserved under scalar extension}.
\end{proof}

We want to show that the scalar extension in Definition \ref{defi: scalar extension of derived category of dga} is compatible with the base change of schemes. In more details, let $X$ be a quasi-compact, separated scheme over $k$ and $\rD(X)$ be the derived category of complexes of quasi-coherent $\mathcal{O}_X$-modules. Let $\mathcal{E}$ be a compact generator of $\rD(X)$ and $A$ be the DG algebra $R\Hom_X(\mathcal{E},\mathcal{E})$. It is well-known  that $\rD(X)\simeq \rD(A)$, see \cite{bondal2003generators} Corollary 3.1.8 and interested readers could obtain an explicit proof using Proposition \ref{prop: equivalence between dg category and module cateogy of dg algebra}.

To study scalar extensions of schemes we first have the following lemma.

\begin{lemma}\label{lemma: scalar extension of morphisms in derived category of schemes}
Let $X$ be as above and $\mathcal{E}\in \Perf(X)$ and $\mathcal{F}\in \rD(X)$. Let $k^{\prime}/k$ be a field extension and $X_{k^{\prime}}:=X\times_k k^{\prime}$ be the base change of $X$ and $p: X_{k^{\prime}}\to X$ be the projection. Then the nature map
$$
\Hom_{\rD(X)}(\mathcal{E},\mathcal{F})\otimes_k k^{\prime}\to \Hom_{\rD(X_{k^{\prime}})}(p^*\mathcal{E},p^*\mathcal{F})
$$
is an isomorphism.
\end{lemma}
\begin{proof}
The proof is the same as that of Lemma \ref{lemma: scalar extension of morphisms in derived category}.
\end{proof}

Then we have the following proposition.

\begin{prop}\label{prop: scalar extension is compatible with base change of schemes}[See \cite{sosna2014scalar} Remark 9]
Let $X$ and $A$ be as above and $k^{\prime}/k$ be a field extension. Let $X_{k^{\prime}}:=X\times_k k^{\prime}$ be the base change of $X$. Then we have
$$
\rD(X_{k^{\prime}})\simeq \rD(A_{k^{\prime}}).
$$
\end{prop}
\begin{proof}Without loss of generality we can assume that $\mathcal{E}$ is an h-injective complex of $\mathcal{O}_X$-modules, hence $A=R\Hom_X(\mathcal{E},\mathcal{E})$ is just the complex $\Hom_{\text{DG}(X)}(\mathcal{E},\mathcal{E})$, where DG$(X)$ is the DG category of complexes of $\mathcal{O}_X$-modules with quasi-coherent cohomologies.

Let $p: X_{k^{\prime}}\to X$ be the natural projection and $\mathcal{E}_{k^{\prime}}=p^*\mathcal{E}$. Then according to \cite{bondal2003generators} Lemma 3.4.1, $\mathcal{E}_{k^{\prime}}$ is also a compact generator of $\rD(X_{k^{\prime}})$. Moreover it is clear that $\Hom_{\text{DG}(X_{k^{\prime}})}(\mathcal{E}_{k^{\prime}},\mathcal{E}_{k^{\prime}})\cong \Hom_{\text{DG}(X)}(\mathcal{E},\mathcal{E})\otimes_k k^{\prime}=A_{k^{\prime}}$.

Let $\mathcal{I}\subset \text{DG}(X_{k^{\prime}})$ be the full pretriangulated DG subcategory consisting of h-injective objects and $z:\mathcal{E}_{k^{\prime}}\to I$ be an h-injective resolution in DG$(X_{k^{\prime}})$. $I$ is a compact generator of $\mathcal{I}$ and we want to apply Proposition \ref{prop: equivalence between dg category and module cateogy of dg algebra} here. It is clear that $\Hom_{\text{DG}(X_{k^{\prime}})}(\mathcal{E}_{k^{\prime}},J)\to \Hom_{\text{DG}(X_{k^{\prime}})}(I,J)$ is a quasi-isomorphism for any $J\in \mathcal{I}$. Moreover, $H^i(\Hom_{\text{DG}(X_{k^{\prime}})}(\mathcal{E}_{k^{\prime}},I))=\Hom_{\rD(X_{k^{\prime}})}(\mathcal{E}_{k^{\prime}},\mathcal{E}_{k^{\prime}}[i])$ and by Lemma \ref{lemma: scalar extension of morphisms in derived category of schemes} the latter is isomorphic to $\Hom_{\rD(X)}(\mathcal{E},\mathcal{E}[i])\otimes_k k^{\prime}=H^i(A_{k^{\prime}})$, hence $A_{k^{\prime}}\to \Hom_{\text{DG}(X_{k^{\prime}})}(\mathcal{E}_{k^{\prime}},I)$ is a quasi-isomorphism. Then by Proposition \ref{prop: equivalence between dg category and module cateogy of dg algebra},
$$
\Hom_{\text{DG}(X_{k^{\prime}})}(\mathcal{E}_{k^{\prime}},-): H^0(\mathcal{I})\to\rD(A_{k^{\prime}})
$$
is an equivalence of triangulated categories. On the other hand $H^0(\mathcal{I})\simeq \rD(X_{k^{\prime}})$ and we finish the proof.
\end{proof}

Now we move on to prove that the scalar extension does not depend on the choice of $A$. To apply Proposition \ref{prop: dg morita equivalence is preserved under scalar extension}, we will need the following important result.

\begin{prop}\label{prop: uniqueness of dga for enhancement of a variety}
Let $X$ be a projective variety and $A$, $B$ be two DG algebras such that $\rD(X)\simeq \rD(A)$ and $\rD(X)\simeq \rD(B)$. Then there exists an $A$-$B$ bimodule $T$ which gives a DG Morita equivalence $(-)\otimes^{\mathbf{L}}_A T: \rD(A)\xrightarrow{\sim} \rD(B)$.
\end{prop}

The proof of Proposition \ref{prop: uniqueness of dga for enhancement of a variety} involves the following concepts and results.

First for a DG algebra (or more generally, a DG category) $A$, as usual we denote the DG category of (right) DG A-modules by $\mathcal{M}_{dg}(A)$ and we use Mod-$A$ to denote the (ordinary) category $Z^0(\mathcal{M}_{dg}(A))$. It is well-known that Mod-$A$ has a projective model structure where weak equivalences are quasi-isomorphisms of chain complexes and  fibrations are degreewise epimorphisms, see \cite{toen2007homotopy} Definition 3.1 or \cite{keller2006differential} Theorem 3.2.

\begin{lemma}\label{lemma: cofibrant and fibrant objects in Mod-A}
For a DG algebra $A$ over a field $k$, the full DG subcategory of $\mathcal{M}_{dg}(A)$ consisting of fibrant and cofibrant objects coincides with h-proj$(A)$.
\end{lemma}
\begin{proof}
Since $k$ is a field, it is easy to see that any DG-modules over $A$ is fibrant. Then we could check by definition that cofibrant objects are exactly h-projective modules. See \cite{barthel2014six} Proposition 1.7.
\end{proof}

It is clear that the homotopy category $H^0(\text{h-proj}(A))$ is equivalent to $\rD(A)$. Now let $A$ and $B$ be as in Proposition \ref{prop: uniqueness of dga for enhancement of a variety} and we know that both h-proj$(A)$ and h-proj$(B)$ give DG enhancements of $\rD(X)$. Now we quote the following important fact about DG enhancement.

\begin{thm}\label{thm: uniqueness of dg-enhancement}[\cite{lunts2010uniqueness} Corollary 7.8]
Let $X$ be a quasi-projective scheme and $\rD(X)$ be the derived category of complexes of quasi-coherent sheaves. Then $\rD(X)$ has a unique DG enhancement, i.e. for two DG enhancement $\mathcal{C}$ and $\mathcal{D}$ there exists a quasi-functor  $\phi: \mathcal{C} \to \mathcal{D}$ which induces an equivalence between their homotopy categories.
\end{thm}
\begin{proof}
See \cite{lunts2010uniqueness} Corollary 7.8.
\end{proof}

\begin{proof}[The proof of Proposition \ref{prop: uniqueness of dga for enhancement of a variety}]
Since both h-proj$(A)$ and h-proj$(B)$ give DG enhancements of $\rD(X)$, we obtain a quasi-functor $\phi: \text{h-proj}(A) \to \text{h-proj}(B)$ by Theorem \ref{thm: uniqueness of dg-enhancement}. A priori $\phi$ is given by a $\text{h-proj}(A)$-$\text{h-proj}(B)$ DG bimodule. But since $\phi$ induces an equivalence between homotopy categories, it is  continuous in the sense of \cite{toen2007homotopy}  Section 7. Therefore by \cite{toen2007homotopy} Corollary 7.6, $\phi$ is given by an $A$-$B$ bimodule $T$.
\end{proof}

The following result, which generalizes Proposition \ref{prop: scalar extension is compatible with base change of schemes}, shows that the scalar extension in Definition \ref{defi: scalar extension of derived category of dga} is compatible with the base change in algebraic geometry.
\begin{coro}\label{coro: scalar extension of any dga of a variety give the base change}
Let $X$ be a projective variety over a field $k$ and $A$ be any DG algebra over $k$ such that $\rD(X)\simeq \rD(A)$. Then for a field extension $k^{\prime}/k$ we have
$$
\rD(X_{k^{\prime}})\overset{\sim}{\to}\rD(A_{k^{\prime}}).
$$
\end{coro}
\begin{proof}
It is a direct corollary of Proposition \ref{prop: scalar extension is compatible with base change of schemes} and Proposition \ref{prop: uniqueness of dga for enhancement of a variety}.
\end{proof}

\subsection{Scalar extensions of categorical resolutions}\label{section: Scalar extensions of categorical resolutions}
In this subsection we discuss scalar extensions of categorical resolutions. First we recall our definition of algebraic categorical resolution, see Definition \ref{defi: categorical resolution our form} above.

\begin{defi}\label{defi: categorical resolution our form rephrase}
Let $X$ be a projective variety over a field $k$. Then an algebraic categorical resolution of $X$ is a triple $(A,B,T)$ where $A$ is a DG algebra such that $\rD(X) \simeq \rD(A)$, $B$ is a smooth DG algebra and $T$ is an
$A$-$B$ bimodule such that
\begin{enumerate}
\item $H^i(A)\to \Hom_{\rD(B)}(T,T[i])$ is an isomorphism for any $i\in \mathbb{Z}$;
\item $T$ defines a compact object in $\rD(B)$;
\item $\bigoplus_i \Hom_{\rD(B)}(T,B[i])$ is finite dimensional.
\end{enumerate}
\end{defi}

Now we define the  scalar extension of a categorical resolution.

\begin{defi}\label{defi: scalar extension of categorical resolutions}
Let $X$ be a projective variety over a field $k$ and $(A,B,T)$ be an algebraic categorical resolution. Let $k^{\prime}/k$ be a field extension. Then the scalar extension of $(A,B,T)$ is given by $(A_{k^{\prime}}, B_{k^{\prime}}, T_{k^{\prime}})$.
\end{defi}

We need to prove that $(A_{k^{\prime}}, B_{k^{\prime}},T_{k^{\prime}})$ in Definition \ref{defi: scalar extension of categorical resolutions} really gives an algebraic categorical resolution of $X_{k^{\prime}}$.

\begin{prop}\label{prop: scalar extension of a categorical resolution is a categocial resolution}
Let $X$ be a projective variety over a field $k$. If  $(A,B,T)$ is an algebraic categorical resolution of $X$, then $(A_{k^{\prime}},B_{k^{\prime}},T_{k^{\prime}})$ is an algebraic categorical resolution of $X_{k^{\prime}}$.
\end{prop}
\begin{proof}
First of all, we know that $B_{k^{\prime}}$ is smooth since $B$ is smooth. Moreover by Corollary \ref{coro: scalar extension of any dga of a variety give the base change} we know that $\rD(X_{k^{\prime}})\simeq \rD(A_{k^{\prime}})$.

 Then we need to show that $(A_{k^{\prime}},B_{k^{\prime}},T_{k^{\prime}})$ satisfies Condition (1), (2), (3) in Definition \ref{defi: categorical resolution our form rephrase}.  Condition (2) follows from Lemma \ref{lemma: characterization of compact objects} and then Condition (1) and (3) are consequences of Lemma \ref{lemma: scalar extension of morphisms in derived category}.
\end{proof}

\section{Application: full exceptional collections of categorical resolutions}\label{section: full exceptional collections of categorical resolutions}
In \cite{wei2016full} the following results has been proved.

\begin{thm}\label{thm: full exceptional collection alg closed field}[\cite{wei2016full} Theorem 4.9]
Let $X$ be a projective curve over an algebraically closed field $k$. Let $(\mathscr{T},\pi^*,\pi_*)$ be a categorical resolution of $X$ (in the sense of Definition \ref{defi: categorical resolution}). If the geometric genus of $X$ is $\geq 1$, then $\mathscr{T}^c$ cannot have a full exceptional collection.

Actually $X$ has a categorical resolution which admits a full exceptional collection if and only if the geometric genus of $X$ is $0$.
\end{thm}

\begin{thm}\label{thm: tilting object alg closed field}[\cite{wei2016full} Theorem 4.10]
Let $X$ be a projective curve over an algebraically closed field $k$ of geometric genus $\geq 1$   and  $(\mathscr{T},\pi^*,\pi_*)$ be a categorical resolution of $X$. Then $\mathscr{T}^c$ cannot have a tilting object, moreover there cannot be a finite dimensional $k$-algebra $\Lambda$ of finite global dimension such that
$$
\mathscr{T}^c\simeq \rD^b(\Lambda\text{-mod}).
$$
\end{thm}

\begin{rmk}
The proof of Theorem \ref{thm: full exceptional collection alg closed field} depends on a careful study of the Picard group of $X$ and the Grothendieck group of the triangulated categories. In particular it involves the fact that Pic$(X)$ is not finitely generated if $X$ is a projective curve with geometric genus $\geq 1$ over an algebraically closed field $k$.

However, if the base field $k$ is not algebraically closed, then  Pic$(X)$ may be finitely generated even if the geometric genus of $X$ is $\geq 1$. Therefore the proof of Theorem \ref{thm: full exceptional collection alg closed field} in \cite{wei2016full} does not work if $k$ is not algebraically closed. See \cite{wei2016full} Remark 9.
\end{rmk}

In order to generalized Theorem \ref{thm: full exceptional collection alg closed field} to curves over non-algebraically closed fields, we need to seek a different proof. The key fact is the following lemma.

\begin{lemma}\label{lemma: scalar extension and full exc collection}
Let $B$ be a DG algebra over a field $k$ and assume Perf$(B)$ has a full exceptional collection $\langle E_1,\ldots, E_n\rangle$. Then for any field extension $k^{\prime}/k$, $\langle (E_1)_{k^{\prime}},\ldots, (E_n)_{k^{\prime}} \rangle$  is a full exceptional collection of Perf$(B_{k^{\prime}})$.
\end{lemma}
\begin{proof}
By Lemma \ref{lemma: scalar extension of morphisms in derived category} we know $\Hom_{\rD(B_{k^{\prime}})}((E_i)_{k^{\prime}},(E_j)_{k^{\prime}})=\Hom_{\rD(B)}(E_i,E_j)\otimes k^{\prime}$ therefore it is clear that $\langle (E_1)_{k^{\prime}},\ldots, (E_n)_{k^{\prime}} \rangle$  is an  exceptional collection of Perf$(B_{k^{\prime}})$.

To show that $\langle (E_1)_{k^{\prime}},\ldots, (E_n)_{k^{\prime}} \rangle$ is full, it is sufficient to show that $B_{k^{\prime}}$ is contained in the triangulated subcategory generated by $(E_1)_{k^{\prime}},\ldots, (E_n)_{k^{\prime}}$. We know that $B$ is contained in the triangulated subcategory generated by $E_1,\ldots, E_n$, i.e. $B$ could be obtained from $E_1,\ldots, E_n$ by taking exact triangles and shifts finitely many times. Tensoring all the exact triangles with $k^{\prime}$ we see that $B_{k^{\prime}}$ is contained in the triangulated subcategory generated by $(E_1)_{k^{\prime}},\ldots, (E_n)_{k^{\prime}}$.
\end{proof}

Then we could generalize Theorem \ref{thm: full exceptional collection alg closed field} to curves over non-algebraically closed fields.

\begin{thm}\label{thm: full exceptional collection non-alg closed field}
Let $X$ be a projective curve over a  field $k$. Then $X$ has a categorical resolution which admits a full exceptional collection if and only if the geometric genus of $X$ is $0$.
\end{thm}
\begin{proof}
The "if" part could be obtained by an explicit construction of a categorical resolution. In fact the proof is exactly the same as that of \cite{wei2016full} Proposition 4.1.

Now assume $X$ has geometric genus $\geq 1$ and there exists a categorical resolution $(\mathscr{T},\pi^*,\pi_*)$ such that $\mathscr{T}^c$ admits a full exceptional collection. Then $\mathscr{T}$ has a compact generator hence by Proposition \ref{prop: a categorical resolution gives a categorical resolution of dga} we have an algebraic categorical resolution $(A,B,T)$ of $X$ and  Perf$(B)\simeq \mathscr{T}^c$ has a full exceptional collection.

Let $\bar{k}$ be the algebraic closure of $k$ and $X_{\bar{k}}$ be the base change. It is clear that $X_{\bar{k}}$ also has geometric genus $\geq 1$. Moreover by Proposition \ref{prop: scalar extension of a categorical resolution is a categocial resolution}, $(A_{\bar{k}},B_{\bar{k}},T_{\bar{k}})$ is a categorical resolution of $X_{\bar{k}}$. By Lemma \ref{lemma: scalar extension and full exc collection} we know Perf$(B_{\bar{k}})=\rD(B_{\bar{k}})^c$ has a full exceptional collection, which is contradictory to Theorem \ref{thm: full exceptional collection alg closed field}.
\end{proof}

We also have the generalization of \cite{wei2016full} Theorem 4.10 to non-algebraically closed base field.
\begin{thm}\label{thm: tilting object non-alg closed field}
Let $X$ be a projective curve with geometric genus $\geq 1$ over a  field $k$ and  $(\mathscr{T},\pi^*,\pi_*)$ be a categorical resolution of $X$. Then $\mathscr{T}^c$ cannot have a tilting object, moreover there cannot be a finite dimensional $k$-algebra $\Lambda$ of finite global dimension such that
$$
\mathscr{T}^c\simeq \rD^b(\Lambda\text{-mod}).
$$
\end{thm}
\begin{proof}
The proof is similar to that of Theorem \ref{thm: full exceptional collection non-alg closed field} and is left to the readers.
\end{proof}

\bibliography{basefieldextension}{}

\begin{thebibliography}{10}

\bibitem{barthel2014six}
Tobias Barthel, J.~P. May, and Emily Riehl.
\newblock Six model structures for {DG}-modules over {DGA}s: model category
  theory in homological action.
\newblock {\em New York J. Math.}, 20:1077--1159, 2014.

\bibitem{bokstedt1993homotopy}
Marcel B{\"o}kstedt and Amnon Neeman.
\newblock Homotopy limits in triangulated categories.
\newblock {\em Compositio Math.}, 86(2):209--234, 1993.

\bibitem{bondal2003generators}
A.~Bondal and M.~van~den Bergh.
\newblock Generators and representability of functors in commutative and
  noncommutative geometry.
\newblock {\em Mosc. Math. J.}, 3(1):1--36, 258, 2003.

\bibitem{hall2015algebraic}
Jack Hall and David Rydh.
\newblock Algebraic groups and compact generation of their derived categories
  of representations.
\newblock {\em Indiana Univ. Math. J.}, 64(6):1903--1923, 2015.

\bibitem{keller2006differential}
Bernhard Keller.
\newblock On differential graded categories.
\newblock In {\em International {C}ongress of {M}athematicians. {V}ol. {II}},
  pages 151--190. Eur. Math. Soc., Z\"urich, 2006.

\bibitem{kuznetsov2008lefschetz}
Alexander Kuznetsov.
\newblock Lefschetz decompositions and categorical resolutions of
  singularities.
\newblock {\em Selecta Math. (N.S.)}, 13(4):661--696, 2008.

\bibitem{kuznetsov2014categorical}
Alexander Kuznetsov and Valery~A. Lunts.
\newblock Categorical resolutions of irrational singularities.
\newblock {\em Int. Math. Res. Not. IMRN}, (13):4536--4625, 2015.

\bibitem{lunts2010categorical}
Valery~A. Lunts.
\newblock Categorical resolution of singularities.
\newblock {\em J. Algebra}, 323(10):2977--3003, 2010.

\bibitem{lunts2010uniqueness}
Valery~A. Lunts and Dmitri~O. Orlov.
\newblock Uniqueness of enhancement for triangulated categories.
\newblock {\em J. Amer. Math. Soc.}, 23(3):853--908, 2010.

\bibitem{lunts2016new}
Valery~A. Lunts and Olaf~M. Schn\"urer.
\newblock New enhancements of derived categories of coherent sheaves and
  applications.
\newblock {\em J. Algebra}, 446:203--274, 2016.

\bibitem{rizzardo2017note}
Alice Rizzardo and Michel Van~den Bergh.
\newblock A note on non-unique enhancements.
\newblock {\em arXiv preprint arXiv:1701.00830}, 2017.

\bibitem{rouquier2008dimensions}
Rapha{\"e}l Rouquier.
\newblock Dimensions of triangulated categories.
\newblock {\em J. K-Theory}, 1(2):193--256, 2008.

\bibitem{sosna2014scalar}
Pawel Sosna.
\newblock Scalar extensions of triangulated categories.
\newblock {\em Appl. Categ. Structures}, 22(1):211--227, 2014.

\bibitem{stacks-project}
The {Stacks Project Authors}.
\newblock \itshape stacks project.
\newblock \url{http://stacks.math.columbia.edu}, 2016.

\bibitem{toen2007homotopy}
Bertrand To{\"e}n.
\newblock The homotopy theory of {$dg$}-categories and derived {M}orita theory.
\newblock {\em Invent. Math.}, 167(3):615--667, 2007.

\bibitem{toen2007moduli}
Bertrand To{\"e}n and Michel Vaqui{\'e}.
\newblock Moduli of objects in dg-categories.
\newblock {\em Ann. Sci. \'Ecole Norm. Sup. (4)}, 40(3):387--444, 2007.

\bibitem{wei2016full}
Zhaoting Wei.
\newblock The full exceptional collections of categorical resolutions of
  curves.
\newblock {\em J. Pure Appl. Algebra}, 220(9):3332--3344, 2016.

\end{thebibliography}
\bibliographystyle{plain}
\end{document}